\documentclass[reqno,oneside,12pt,hidelinks]{amsart}
\usepackage{amssymb,amsmath,amsthm,amsxtra,calc,bm, color}
\usepackage{enumerate}
\usepackage[margin=.95in]{geometry}

\usepackage[scr=boondoxo]{mathalfa}
\usepackage{mathtools}

\usepackage{mleftright}
\mleftright

\usepackage{nccmath}
\usepackage{cases}

\usepackage{calc}
\usepackage{graphicx}
\usepackage{hyperref}
\usepackage[final]{microtype}

\renewcommand{\(}{\left(}
\renewcommand{\)}{\right)}
\newcommand{\spmod}[1]{\ensuremath{\,(#1)}}

\renewcommand{\|}{\big |}

\def\Z{\mathbb{Z}}
\def\Q{\mathbb{Q}}
\def\R{\mathbb{R}}
\def\H{\mathbb{H}}

\def\C{\mathbb{C}}

\def\F{\mathbb{F}}

\def\SL{{\rm SL}}
\def\PSL{{\rm PSL}}
\def\GL{{\rm GL}}

\newcommand{\pfrac}[2]{\left(\frac{#1}{#2}\right)}
\newcommand{\pmfrac}[2]{\left(\mfrac{#1}{#2}\right)}
\newcommand{\ptfrac}[2]{\left(\tfrac{#1}{#2}\right)}
\newcommand{\pMatrix}[4]{\left(\begin{matrix}#1 & #2 \\ #3 & #4\end{matrix}\right)}

\renewcommand{\pmatrix}[4]{\left(\begin{smallmatrix}#1 & #2 \\ #3 & #4\end{smallmatrix}\right)}
\renewcommand{\bar}[1]{\overline{#1}}

\newcommand{\Sh}{\mathcal{S}}

\renewcommand{\c}{\mathfrak{c}}

\renewcommand{\sl}{\big| }
\DeclareMathOperator{\new}{new}

\def\ep{\varepsilon}

\newtheorem{theorem}{Theorem}[section]
\newtheorem{lemma}[theorem]{Lemma}

\newtheorem{proposition}[theorem]{Proposition}

\theoremstyle{remark}
\newtheorem*{remark}{Remark}

\newtheorem*{definition}{Definition}
\numberwithin{equation}{section}

\newcommand{\Frob}{\operatorname{Frob}}
\mathtoolsset{showonlyrefs}

\newcommand{\Gal}{\operatorname{Gal}}

\title[Quadratic congruences for cusp forms with eta multiplier]{Quadratic congruences for half-integral weight cusp forms with the eta multiplier}

\date{\today}

\author{Robert Dicks}
\address{Department of Mathematics\\
Clemson University\\
Clemson, SC 29634} 
\email{rdicks@clemson.edu}

\begin{document}

\begin{abstract}
Let $\ell \geq 5$ be a prime, and let $\nu_\eta$ denote the Dedekind eta multiplier.
For an odd integer $r$, and a real Dirichlet character $\psi$, recent work of Ahlgren, Andersen, and the author showed that quadratic congruences modulo $\ell$ hold for a wide range of half-integral weight cusp forms with multiplier $\psi\nu_\eta^r$, vastly generalizing certain congruences discovered by Atkin for the partition function.

In this paper, we show that such congruences hold when $\psi$ is an arbitrary character. Our methods rely on the theory of modular Galois representations. 
For primes $\ell \geq 5$, the core of our work is the study of modular Galois representations modulo $\ell$ attached to integer-weight eigenforms with arbitrary Nebentypus whose images are large in a precise sense.
Our key new result is that, given a finite set of such representations and $\gamma \in \SL_2(\F_\ell)$, there exists $\sigma \in \Gal(\bar{\Q}/\Q(\zeta_\ell))$ whose images under the representations are in the conjugacy class of $\gamma^2$.
\end{abstract}

\maketitle

\section{Introduction}
For a positive integer $n$, the partition function $p(n)$ counts the number of ways of writing $n$ as a sum of a nonincreasing sequence of positive integers (where we agree by convention that $p(0)=1$ and $p(n)=0$ whenever $n \not \in \Z_{\geq0}$). The study of the arithmetic properties of $p(n)$ has been deeply influential for the arithmetic of modular forms. 
For example, the Ramanujan congruences \cite{Ramanujan}

\begin{equation} \label{eq:ram-cong}
    p\pfrac{\ell n+1}{24} \equiv 0 \pmod{\ell}, \qquad \ell=5,7,11,
\end{equation}
anticipated arithmetic phenomena which occur for the coefficients of many weakly holomorphic modular forms.

After the work of Ramanujan, many other congruences for $p(n)$ were found, the most relevant of which for the present paper are Atkin's congruences, which we now describe.
For several primes $\ell \leq 31$, Atkin \cite[eq.~$(52)$]{Atkin2} found examples of congruences of the form
\begin{equation} \label{eq:atkin-cong}
    p\pfrac{\ell Q^2 n+\beta}{24} \equiv 0 \pmod{\ell} \quad \text{ if } \ \ \pfrac nQ = \ep_Q,
\end{equation}
where $Q$ is a prime different from $\ell$, $\beta$ is an integer, and $\ep_Q \in \{\pm 1\}$.
Recent work of Ahlgren, Allen, and Tang \cite{AAT} used modular Galois representations to show that, for every prime $\ell \geq 5$, there exists a positive density set of primes $S$ such that congruences of the form \eqref{eq:atkin-cong} hold for $Q \in S$ . 

The techniques of Ahlgren, Allen, and Tang are Galois theoretic and use work of Yang \cite{Yang} on the usual Shimura correspondence for half-integral weight cusp forms on $\SL_2(\Z)$. More precisely, let $f(z)$ be such a form which transforms with multiplier $\nu^r_\eta$, where $\nu_\eta$ is the multiplier for the Dedekind eta function (see \eqref{etamultiplier} below), and $r$ is an integer with $(r,6)=1$. Then $f(24z)$ is a half-integral weight form in the sense of Shimura. Yang obtained precise information on the images of the Shimura lifts applied to $f(24z)$ (he proves a similar result when $(r,6)=3$).  

For many values $N \in \Z^{+}$,
Ahlgren, Andersen, and the author \cite{AAD} developed a precise Shimura correspondence for cusp forms of half-integral weight with eta multiplier of level $N$. 
The arithmetic consequences of that work generalize congruences like \eqref{eq:atkin-cong} to a wide class of forms in $S_{\lambda+\frac{1}{2}}(N,\psi\nu_\eta^r)$ (see Section $2$ for details on modular forms and related notation and terminology).
Here,  $\lambda\in \Z^{+}$, $r$ is an odd integer and $\psi$ is a real Dirichlet character modulo $N$.
In this paper, we prove similar results when $\psi$ is an \emph{arbitrary} Dirichlet character modulo $N$.

Stating our results requires some notation.  Let $\lambda, N \in \Z^{+}$ and let $r$ be an odd integer. 
Let $\psi$ be a Dirichlet character modulo $N$.
If $(N,6)=1$, then we let 
$S^{\operatorname{new}2,3}_{2\lambda}(6N, \psi^2, \epsilon_2,\epsilon_3)$
consist of the cusp forms of weight $2\lambda$ on $\Gamma_0(6N)$ with character $\psi^2$ which are new at $2$ and $3$ and with Atkin-Lehner eigenvalues $\epsilon_2$ and $\epsilon_3$, respectively, at those primes. We assume the forms satisfy certain integrality conditions which we have suppressed from the notation (see Section $2$ for details). We make a similar definition for $S^{\operatorname{new}2}_{2\lambda}(2N, \psi^2, \epsilon_2)$ when $N$ is odd. Let $\(\frac{\bullet}{\bullet}\)$ denote the extended quadratic symbol and define
\[
\ep_{2, r, \psi}:=-\psi(2)\pmfrac{8}{r/(r,3)}, \ \ \ \ 
\ep_{3, r, \psi}:=-\psi(3)\pmfrac{12}{r},
\]

In weight $3/2$, we need to avoid unary theta series by virtue of the Shimura correspondence that we use. When $(r,6)=1$ we let $S_{3/2}^c(N,\psi\nu_\eta^r)$ consist of the forms $F \in S_{3/2}(N,\psi\nu_\eta^r)$ which satisfy
\begin{equation}\label{eq:orthogtheta}
    \left\langle F\sl V_{24}, G \right\rangle = 0 \ \text{ for all theta functions }\  G\in S_{\frac 32} \left(576N, \psi \ptfrac{-1}{\bullet}^{\frac{r+1}{2}}\ptfrac{12}\bullet \nu_\theta^{3}\right),
\end{equation}
where $\langle \cdot, \cdot \rangle$ is the usual Petersson inner product, $V_{24}$ is the operator defined in \eqref{Voperator} and $\nu_\theta$ is the multiplier for the usual theta function. We make the same definition for $S^{c}_{3/2}(N,\psi\nu^r_{\eta})$ when $(r,6)=3$ with $S_{\frac 32} \left(576N, \psi \ptfrac{-1}{\bullet}^{\frac{r+1}{2}}\ptfrac{12}\bullet \nu_\theta^{3}\right)$ replaced by $S_{\frac 32}\(64N,  \psi\ptfrac{-1}\bullet^{\lambda+\frac{r-1}2} \nu_{\theta}^{2\lambda+1}\)$.

Crucial to our first result is the notion of suitability. For each newform $f$ in the spaces under our consideration, this is a technical hypothesis on the modulo $\ell$ reduction $\bar{\rho}_f$   of the $\ell$-adic Galois representation $\rho_f$ associated with $f$  (see Theorem~\ref{bigGalThm}
 below); our definition  specializes to that of \cite[$\mathsection$ $1$]{AAD} when $\psi$ is real. 
 
 \begin{definition}
\label{def:suitable}
Suppose that $\ell \geq 5$ is prime and that $r$ is an odd integer. Suppose that $N$ is a squarefree, odd, positive integer with $\ell \nmid N$. Assume that $3 \nmid N$ if $3 \nmid r$. Suppose that $\psi$ is a Dirichlet character modulo $N$, and  let $k$ be a positive even integer.

If $(r,6)=1$, then 
we say that the pair $(k,\ell)$ is \emph{suitable} for the triple $(N,\psi,r)$ if
for every newform $f \in S^{\text{new } 2,3}_{k}(6N,\psi^2,\ep_{2,r,\psi},\ep_{3,r,\psi})$, the image of $\bar{\rho}_f$ contains a conjugate of 
$\SL_2(\mathbb{F}_\ell)$. If $(r,6)=3$, then we make a similar definition for newforms in $S^{\new 2}_{k}(2N,\psi^2,\ep_{2,r,\psi})$.
\end{definition}
 \begin{remark}
 For $\lambda \in \Z^{+}$, we will show that $(2\lambda, \ell)$ is suitable for every triple $(N, \psi, r)$ if the following conditions hold (see Proposition~\ref{3.3 analogue} below):
 \begin{itemize}
 \item
 $\ell > 10\lambda-4$.
 \item
 $\psi(n) \neq -1$ for all $n \in \Z$.
 \item
 $\ell \nmid \phi(N)$.
 \item
  $2^{2\lambda-1} \not \equiv 2^{\pm 1} \pmod{\ell}$.
 \end{itemize}
\end{remark}

Finally, for $z$ in the complex upper half-plane, we define $q=q(z)=e^{2 \pi i z}$.
Our main result is as follows. 
\begin{theorem}\label{thm:cong1}
Suppose that  $\ell \geq 5$ is prime and that $r$ is an odd integer.
Suppose that $m$ and $\lambda$ are positive integers.
Suppose that $N$ is a squarefree, odd, positive integer with $\ell \nmid N$. Assume that $3 \nmid N$ if $3 \nmid r$. 
Let $\psi$ be a Dirichlet character modulo $N$.
Suppose that 
\begin{equation}
F(z) =\displaystyle \sum_{n \equiv r \spmod{24}} a(n)q^{\frac{n}{24}} \in S_{\lambda+\frac{1}2}(N, \psi\nu_\eta^{r})
\end{equation}
with $(2\lambda,\ell)$ suitable for $(N,\psi,r)$, and if $\lambda=1$, suppose further that $F \in S^{c}_{{3}/2}(N,\psi\nu_\eta^r)$. 
Then there is a positive density set $S$ of primes such that if $p \in S$, then $p \equiv 1 \pmod{\ell^m} $ and
\begin{equation}
a(p^2n) \equiv 0 \pmod{\ell^m}  \ \ \ \ \text{ if } \ \ \ \(\frac{n}{p}\)=\begin{cases} 
-\(\frac{-1}{p}\)^{\frac{r-1}{2}}\psi(p) & \text{if } 3 \nmid r, \\
-\(\frac{-3}{p}\)\(\frac{-1}{p}\)^{\frac{r-1}{2}}\psi(p) & \text{if }  3\mid r.
\end{cases}
\end{equation}
\end{theorem}
\begin{remark}
In the theorem above and the theorems which follow, our definition of density is that of natural density. The set $S$ depends on $F$, $\ell$, and $m$ and is frobenian in the sense of Serre (see \cite[$\mathsection$ $3.3$]{serre-NXp} and the discussion at the end of \cite[$\mathsection$ $1$]{AAT}).
\end{remark}
Our next result does not rely on suitability; it is an analogue of \cite[Theorem $1.4$]{AAD}.

 \begin{theorem}\label{thm:cong2}
 Suppose that $\ell \geq 5$ is prime and that $r$ is an odd integer.
 Let $m$ and $\lambda$ be positive integers.
 Suppose that $N$ is a squarefree, odd, positive integer with $\ell \nmid N$. Assume that $3 \nmid N$ if $3 \nmid r$.
 Let $\psi$ be a Dirichlet character modulo $N$. Suppose that there exists  $a \in \Z$ with the property that
 \begin{equation}\label{Hasse}
 2^{a} \equiv -2 \pmod\ell . 
 \end{equation}
   Let 
   \begin{equation}
   F(z) =\displaystyle \sum_{n \equiv r \spmod{24}} a(n)q^{\frac{n}{24}} \in S_{\lambda+\frac{1}2}(N, \psi\nu_\eta^{r}),
  \end{equation}
and if $\lambda=1$, suppose further that $F \in S^{c}_{{3}/2}(N,\psi\nu_\eta^r)$.
 Then there is a positive density set $S$ of primes such that if $p \in S$, then $p \equiv -2 \pmod{\ell^{m}}$ and for some $\ep_{p} \in \{\pm 1\}$ we have
 \begin{equation}
 a(p^2n) \equiv 0 \pmod{\ell^{m}} \ \ \ \text{ if } \ \ \ \(\frac{n}{p}\)=\ep_{p}.
 \end{equation}
 \end{theorem}
\begin{remark}
The value of $\ep_{p}$ can be explicitly calculated using Theorem~\ref{4.2 analogue} below.  By a result of Hasse \cite{Hasse},  the proportion of primes satisfying \eqref{Hasse} is $17/24$.
\end{remark}

The main new input in our work is a result which roughly speaking shows that one can find $\sigma \in \Gal(\bar{\Q}/\Q(\zeta_\ell))$ for which the images of $\sigma$ under the modulo $\ell$ Galois representations which we consider are in the same conjugacy class. We mention a simple special case (the precise result is Proposition $3.5$).
\begin{proposition}\label{specialcase}
Let $N \in \Z^{+}$ and $\psi$ be a Dirichlet character modulo $N$. Assume that $f_1,...,f_s \in S_{k}(N,\psi)$ are normalized eigenforms such that for each $i \in \{1,\dots,s\}$, $\bar{\rho}_{f_{i}}(G_\Q)$ contains a conjugate of $\SL_2(\F_\ell)$.
 Then for every $\gamma \in \SL_2(\F_\ell)$, there exists $\sigma \in \Gal(\bar{\Q}/\Q(\zeta_{\ell}))$ such that, if $i \in \{1,...,s\}$, then $\bar{\rho}_{f_i}(\sigma^2)$ is conjugate to $\gamma^2$.
 \end{proposition}
 Ahlgren, Allen, and Tang proved a more precise version of Proposition~\ref{specialcase} when $\psi$ is trivial (see \cite[Proposition $3.8$]{AAT}). At crucial steps in their argument, they are confronted with the problem of identifying several Galois characters modulo $\ell$. By virtue of working in spaces with trivial Dirichlet character, the Galois characters are easy to identify: they are either trivial or $\omega_\ell^{(\ell-1)/2}$, where $\omega_\ell$ is the mod $\ell$ cyclotomic character.

The main obstruction to generalizing \cite[Proposition $3.8$]{AAT} is that the corresponding Galois characters which come up in the analysis are difficult to identify. The crucial insight which allows us to prove Proposition~\ref{nonuniformconj} is that we can prove the result ``up to a constant factor" without needing to know which Galois characters arise.  We are then able to leverage the theory of modular Galois representations to show that the constant factors which occur are extremely well-behaved. In fact, \cite[Proposition 3.8]{AAT} is true ``up to sign" in our context, which allows us to  recover the result for matrices $\gamma^2$, where $\gamma\in \SL_2(\F_\ell)$. This is sufficient for our arithmetic applications.

The paper is organized as follows. In Section $2$, we give background on modular forms and Galois representations. Section $3$ is the core of the paper; we prove preliminary results on Galois representations and congruences for newforms which we need for proving Theorem~\ref{thm:cong1}. This section also features sufficient conditions for suitability. 
Most of our work is devoted to building tools to prove the main technical result needed for proving Theorem~\ref{thm:cong1}. In particular, we need to prove that the Hecke operators act diagonally on the newforms in the spaces relevant for us  with eigenvalue $-1$ modulo arbitrary powers of $\ell$. Suitability is crucial for this result; it does not feature in the corresponding technical result needed for proving Theorem~\ref{thm:cong2}. 
In Section $4$, we prove Theorems~\ref{thm:cong1}-\ref{thm:cong2}. Our methods make crucial use of the Shimura correspondence developed in \cite{AAD} to translate our congruences for newforms into congruences for half-integral weight forms.
 
 \section{Background}
 We primarily follow the exposition in  \cite[$\mathsection$ $2$]{AAT} and \cite[$\mathsection$ $3$]{AAD}. Throughout this section, we assume that $\ell \geq 5$ is prime. For $z$ in the upper half-plane $\H$, we let $q=q(z)=e^{2 \pi i z}$.
 If $f$ is a function $\H$,  $k\in \frac12\Z$,  and $\gamma=\pmatrix abcd\in \GL_2^+(\R)$, then
we define

\[
\(f\|_k\gamma\)(z)=(\det\gamma)^\frac k2(cz+d)^{-k}f(\gamma z).
\]

Let $A \subseteq \C$ be a subring.
If $\nu$ is a multiplier system on $\Gamma_0(N)$, 
we denote by $M_{k}(N,\nu,A)$ (respectively $S_{k}(N,\nu,A)$)
the $A$-module of modular forms (respectively cusp forms) 
of weight $k$ and multiplier 
$\nu$ on $\Gamma_0(N)$ whose Fourier coefficients are in $A$.
When $\nu=1$ or $A$ is the subring of algebraic numbers that are integral at all of the primes above $\ell$, we omit them from the notation.
Forms in these spaces satisfy the transformation law
\[
f \sl_k \gamma= \nu(\gamma)f \ \ \ \text{ for } \ \ \ \gamma = \left(\begin{matrix}a & b \\c & d\end{matrix}\right) \in \Gamma_0(N)
\]
and the appropriate conditions at the cusps of $\Gamma_0(N)$. 

We define the eta function by
\[
\eta(z):= q^{\frac{1}{24}}\prod_{n=1}^{\infty}(1-q^{n}).
\]
The eta function has a multiplier $\nu_{\eta}$ satisfying

\[
\eta(\gamma z)=\nu_{\eta}(\gamma)(cz+d)^{\frac{1}{2}}\eta(z), \ \ \ \ \ \gamma= \left(\begin{matrix}a & b \\c & d\end{matrix}\right) \in \SL_{2}(\Z);
\]
throughout, we choose the principal branch of the square root. For $c > 0$,
we have the formula 
  \cite[~$\mathsection$$4.1$]{Knopp}

\begin{equation}\label{etamultiplier}
\nu_{\eta}(\gamma)=
 \begin{cases} 
 \(\frac{d}{c}\)e\(\frac{1}{24}((a+d)c-bd(c^2-1)-3c )\) & \text{if } c \text{ is odd,} \\
\(\frac{c}{d}\)e\(\frac{1}{24}((a+d)c-bd(c^2-1)+3d-3-3cd)\) & \text{if } c \text{ is even},\\
\nu_\eta(-\gamma)=i\nu_\eta(\gamma).
\end{cases}
\end{equation}

We next recall the $U$ and $V$ operators. For a positive integer $m$, we define them on Fourier expansions by
\begin{equation}
\(\sum_{n=1}^{\infty}a(n)q^{\frac{n}{24}}\) \sl U_{m}:= \sum_{n=1}^{\infty}a(mn)q^{\frac{n}{24}},
\end{equation}

\begin{equation}\label{Voperator}
\(\sum_{n=1}^{\infty}a(n)q^{\frac{n}{24}}\)\sl V_{m}:= \sum_{n=1}^{\infty}a(n)q^{\frac{mn}{24}}.
\end{equation}
If $p$ is a prime and $\psi$ is a  Dirichlet character modulo $N$, then we define the Hecke operator $T_p: S_k(N, \psi)\to S_k(N, \psi)$ via
\begin{equation}\label{eq:inthecke}
T_p=U_p+\psi(p)p^{k-1}V_p.
\end{equation}

If $p\mid N$ and $\psi$ is defined modulo $N/p$, then we denote by $S_k^{\new p}(N, \psi, \C)$ the orthogonal complement (with respect to the Peterson inner product) of the subspace in $S_k(N, \psi, \C)$ generated by 
$S_k\(N/p, \psi, \C\)$ and 
$S_k\(N/p, \psi, \C\)\|V_p$.
If $p\mid N$ and $(p, N/p)=1$, then    the Atkin-Lehner matrix $W^N_p$ is any integral matrix with 
\begin{equation}\label{eq:atkinlehner}
W^N_p =\pMatrix{p\alpha }\delta{N\beta}{p }, \ \ \ \ \alpha,\beta,\delta \in \Z, \ \ \ \ \det(W^N_p)=p.
\end{equation}
If $\chi$ is defined modulo $N/p$ 
then the operator $\|_kW^N_p$ preserves the space $M_k(N, \chi)$, and on this space is independent of the particular choice of matrix \cite[Lemma 2]{Li:1975}.

For a subring $A \subseteq \C$ and a multiplier system $\nu$ on $\Gamma_0(N)$, we set $S_k^{\new}(N,\nu,A) \subseteq S_k(N,\nu,A)$ and $S^{\new p}_k(N,\nu, A) \subseteq S_k(N,\nu, A)$ for the corresponding $A$-submodules of forms with coefficients in $A$ (and apply the same convention regarding algebraic numbers which are integral at all of the primes above $\ell$ and the trivial multiplier system).
Let $N \in \Z^+$ and $\psi$ be a Dirichlet character modulo $N$. Suppose that $f\in S^{\new p}_k(N,\psi)$ is a newform, by which we mean that there exists a level $N'$ with $p\mid N'$ and $N' \mid N$
such that 
$f\in S^{\new}_k(N',\psi)$ and $f$ is  a normalized eigenform of the operators $T_Q$ for $Q\nmid N'$ and $U_Q$ for $Q\mid N'$ (we refer to such a newform as a newform of level $N'$).
Then there is an Atkin-Lehner eigenvalue $\ep_p\in \{\pm 1\}$ such that
$f\sl_k W_p^N=f\sl_k W_p^{N'}=\ep_p f$; if $f=\sum a(n)q^n$ then $\ep_p=-p^{1-k/2}a(p)$ \cite[Corollary $3.2$]{AAD}.
Though we mainly consider newforms, many of the results we use and prove are true for normalized eigenforms in $S_k(N,\psi)$; more precisely, these are the normalized eigenforms of the operators $T_Q$ for $Q \nmid N$ and $U_Q$ for $Q \mid N$.

We recall the definition of the Shimura correspondence for the eta multiplier in \cite{AAD}.  Suppose that $F=\sum a(n)q^{n/24}\in S_{\lambda+1/2}(N, \psi\nu_\eta^{r})$,
where $\lambda\geq 1$;  if $\lambda=1$ suppose further that $F \in S^c_{3/2}(N,\psi\nu_\eta^r)$. For a Dirichlet character $\chi$ modulo $N$, let $L(s,\chi)$ denote the Dirichlet $L$-function.
If $t$ is a positive squarefree 
integer, the $t^{\operatorname{th}}$ 
lift is given by  $\Sh_t(F)=\sum 
b(n)q^n$, where the coefficients 
$b(n)$ are given by 
\begin{equation}\label{eq:shimlift}
\sum_{n=1}^\infty \frac{b(n)}{n^s} = L\(s-\lambda+1,\psi \ptfrac \bullet t\) \sum_{n=1}^\infty \pfrac{12}{n} \frac{a(tn^2)}{n^s}.
\end{equation}
We have $\Sh_t(F)\in S^{\text{new } 2,3}_{2\lambda}(6N,\psi^2,\ep_{2,r,\psi},\ep_{3,r,\psi})$ if $(r,6)=1$ and $\Sh_t(F)\in S^{\text{new } 2}_{2\lambda}(2N,\psi^2,\ep_{2,r,\psi})$ if $(r,6)=3$ \cite[Theorems $1.1-1.2$]{AAD}.
By a standard argument (see e.g. \cite[$(2.13)$]{Scarcity}, we see that
\begin{equation}\label{f0iffShimura0}
F \equiv 0 \pmod{\ell} \iff \Sh_{t}(F) \equiv 0 \pmod{\ell} \ \ \ \text{ for all positive, squarefree $t$}.
\end{equation}

For primes $p \geq 5$, we make use of Hecke operators $T_{p^2}$ for forms of half-integral weight. Suppose that  $(r, 6)=1$ and that  
\begin{equation}
		F(z) = \sum_{n\equiv r\spmod {24}} a(n) q^\frac n{24} \in  M_{\lambda+\frac12}(N,\psi\nu_\eta^r).
	\end{equation}
Then the action of $T_{p^2}$ is given by
\begin{equation}\label{eq:heckedef24}
F \sl T_{p^2}=\sum_{n\equiv r(24)} \(a(p^2n)+\pmfrac{-1}p^\frac{r-1}2\pmfrac{12n}p\psi(p)p^{\lambda-1}a(n)+\psi^2(p)p^{2\lambda-1}a\pmfrac n{p^2}\)q^\frac n{24}.
\end{equation}
These operators preserve the spaces of cusp forms.
Morever, if we set
\begin{equation}\label{chir}
 \chi^{(r)}= \begin{cases} 
\(\frac{-4}{\bullet}\) & \text{if } 3 \mid r, \\
\(\frac{12}{\bullet}\) & \text{if }  3\nmid r,
\end{cases}
\end{equation}
then we have
\begin{equation}\label{equivariance}
\Sh_{t}\(F\sl T_{p^{2}}\)=\chi^{(r)}(p)\Sh_{t}(F)\sl T_{p}.
\end{equation}

We now discuss modular Galois representations.  See \cite{Hida} and \cite{Edixhoven} for more details. 
We begin with some notation. Suppose that $\ell \geq 5$ is prime. Let $k$ be an even integer and $N \in \Z^{+}$ with $\ell \nmid N$. 
Let $\bar{\Q}$ be the algebraic closure of $\Q$ in $\C$. If $p$ is prime, let $\bar{\Q}_{p}$ be a fixed algebraic closure of $\Q_{p}$ and fix an embedding $\iota_{p}:\bar{\Q} \hookrightarrow \bar{\Q}_{p}$. 
 The embedding $\iota_{\ell}$ allows us to view the coefficients of forms in $S_{k}(N)$ as elements of $\bar{\Q}_{\ell}$, and for each prime $p$,
the embedding $\iota_{p}$ allows us to view $G_{p}:=\Gal(\bar{\Q}_{p}/\Q_{p})$ as a subgroup of $G_{\Q}:=\Gal(\bar{\Q}/\Q)$.
 If $I_{p} \subseteq G_{p}$ is the inertia subgroup, we denote arithmetic Frobenius in $G_{p}/I_{p}$ by $\text{Frob}_{p}$. 
 
 For any finite extension $K/\Q$, 
 we set $G_{K}:=\Gal(\bar{K}/K)$ and 
 denote by $\text{Frob}_{p} |_{K}$ the restrictions to $K$ of elements in $\text{Frob}_{p}$.
 For representations $\rho:G_{\Q} \rightarrow \GL_{2}(\bar{\Q}_{\ell})$ which are unramified at $p$ (which is to say that $I_{p} \subseteq \ker{\rho}$), the image of any $\sigma \in G_{p}$ under $\rho$ only depends on $\sigma I_{p} \in G_{p}/I_{p}$; 
  for any such representation, we denote by $\rho(\text{Frob}_{p})$ the image of any element in $\text{Frob}_{p}$ under $\rho$.

  We denote by $\chi_\ell:G_{\Q}\rightarrow \Z^{\times}_{\ell}$ and $\omega_\ell:G_{\Q}\rightarrow \mathbb{F}^{\times}_{\ell}$ the $\ell$-adic and mod $\ell$ cyclotomic characters, respectively.
 We let $\omega_{2},\omega'_{2}:I_{\ell}\rightarrow\mathbb{F}^{\times}_{\ell^{2}}$ denote Serre's fundamental characters of level $2$ (see \cite[~$\mathsection$$2.1$]{DukeSerre}).
 Both characters have order $\ell^{2}-1$, and we have 
$\omega^{\ell+1}_{2}=\omega'^{\ell+1}_{2}=\omega_\ell$. Finally, we mention that if $\psi$ is a Dirichlet character modulo $N$, we may regard it as a Galois character $\psi: \Gal(\bar{\Q}/\Q) \rightarrow \bar{\Z}^\times_\ell$ in the standard way (see e.g. \cite[Chapter $9$ $\mathsection$ $3$]{Diamond-Shurman}).

 The following theorem is due to Deligne, Fontaine, Langlands, Ribet, and Shimura (see \cite[Theorem 2.1]{AAT}). 
\begin{theorem}\label{bigGalThm}
Suppose that $\ell \geq 5$ is prime and that $\psi$ is a Dirichlet character modulo $N$ with conductor $C_\psi$.
Let $f=\sum a(n)q^{n} \in S_{k}(N,\psi)$ be a normalized eigenform. There is a continuous irreducible representation $\rho_{f}:G_{\Q} \rightarrow \GL_2(\bar{\Q}_\ell )$ with semisimple mod $\ell$ reduction $\bar{\rho}_{f}:G_{\Q} \rightarrow \GL_2(\bar{\mathbb{F}}_\ell )$ satisfying the following properties.
\begin{enumerate}
\item
If $p \nmid \ell N$, then $\rho_{f}$ is unramified at $p$ and the characteristic polynomial of $\rho_{f}(\Frob_{p})$ is $X^2-\iota_\ell (a(p))X+\psi(p)p^{k-1}$. This uniquely characterizes $\rho_f$.
\item
If $f \in S^{\operatorname{new }Q}_k(N)$, where $Q$ is a prime with $Q\mid\mid N$ then we have
\begin{equation}
\rho_{f}|_{G_{Q}} \cong \left(\begin{matrix}\chi_\ell  \psi_Q & * \\0 & \psi_Q \end{matrix}\right),
\end{equation}
where $\psi_Q:G_{Q} \rightarrow \bar{\Q}^{\times}_\ell $ is the unramified character with $\psi_Q({\Frob}_{Q})=\iota_\ell (a(Q))$. If we also have $Q \nmid C_\psi$, then $\rho_f | _{I_Q}$ is unipotent. 
If $Q \mid C_\psi$, we have  
\begin{equation}
\rho_{f}|_{I_{Q}} \cong \left(\begin{matrix}\psi & * \\0 & 1 \end{matrix}\right),
\end{equation}
where we have regarded $\psi$ as a Galois character.
\item
Assume that $2 \leq k \leq \ell+1$. 
\begin{itemize}
\item
If $\iota_\ell (a(\ell)) \in \bar{\Z}_\ell ^{\times}$, then $\rho_{f} |_{G_\ell }$ is reducible and we have
\begin{equation}
\rho_{f}|_{I_\ell } \cong \left(\begin{matrix}\chi_\ell ^{k-1} & * \\0 & 1 \end{matrix}\right).
\end{equation}
\item
If $\iota_\ell (a(\ell)) \not \in \bar{\Z}_\ell ^{\times}$, then $\bar{\rho}_{f} |_{G_\ell }$ is irreducible and $\bar{\rho}_{f}|_{I_\ell } \cong \omega^{k-1}_2 \oplus \omega'^{(k-1)}_2$.
\end{itemize}
\end{enumerate}
\end{theorem}

\begin{remark}
The Galois representations depend on the choice of embedding $\iota_\ell :\bar{\Q} \hookrightarrow \bar{\Q}_\ell $, but we have suppressed this from the notation.
\end{remark}
We now introduce notation concerning group representations (see also \cite[$\mathsection$ $3$]{AAT}).
For $\tau \in \Gal(\bar{\F}_\ell/\F_\ell)$
and a homomorphism $\bar{\rho}:G_\Q \rightarrow \GL_2(\bar{\F}_\ell)$, we write ${}^\tau\bar{\rho}:G_\Q \rightarrow \GL_2(\bar{\F}_\ell)$ for the composite of $\bar\rho$ with the automorphism of $\GL_2(\bar{\F}_\ell)$ induced by $\tau$. 
We make a similar definition for a homomorphism $r:G_\Q \rightarrow \text{PGL}_2(\bar{\F}_\ell)$.
For two homomorphisms $r_1, r_2:G_\Q \rightarrow \text{PGL}_2(\bar{\F}_\ell)$, 
we write $r_1 \cong r_2$
if they are conjugate by an element of $\text{PGL}_2(\bar{\F}_\ell)$.
When matrices $A,B \in \GL_2(\bar{\F}_\ell)$ are conjugate by an element of $\GL_2(\bar{\F}_\ell)$,
we write $A \sim B$.
For a normalized eigenform $f \in S_k(N,\psi)$, we let $r_f:G_\Q \rightarrow \text{PGL}_2(\bar{\F}_\ell)$ be the composite of $\bar{\rho}_f:G_\Q \rightarrow \GL_2(\bar{\F}_\ell)$ with the natural projection $\GL_2(\bar{\F}_\ell) \rightarrow \text{PGL}_2(\bar{\F}_\ell)$.

Lastly, for a Galois field extension $L/K$ and a subgroup $H \subseteq \Gal(L/K)$, we denote by $L^H$ the subfield of $L$ which is fixed by $H$. 
 \section{Results on Modular Galois representations}

 \subsection{Preliminary results}
 Throughout this section, $\ell \geq 5$ is prime, $k$ is an even, positive integer and $N$ is a squarefree, positive integer.
 We begin by recording two useful consequences of Theorem $2.1$ which are useful for proving Proposition$~3.3$, which gives sufficient conditions for the condition of suitability mentioned in the introduction.
 
 The following result, an analogue of \cite[Lemma $3.1$]{AAT}, gives precise information on how mod $\ell$ modular Galois representations which are isomorphic up to a twist by a mod $\ell$ Galois character are related. When the weight is low enough, one can constrain the weights of the normalized eigenforms under our consideration.

 \begin{lemma}\label{3.1 analogue}
 Let $\psi$ be a Dirichlet character modulo $N$ which we regard as a Galois character.
 Let $f,g \in S_k(N,\psi)$ be normalized eigenforms such that $f=\sum a(n)q^n$ and $\bar{\rho}_g \cong \bar{\rho}_f \otimes\varphi$, where $\varphi: G_\Q\rightarrow\bar{\F}^\times_\ell$ is a nontrivial continuous character.
 Suppose further that $\ell \nmid \phi(N)$ and that $\psi(\sigma) \neq -1$ for all $\sigma \in G_\Q$. Then
 \begin{enumerate}
 \item
 $\varphi=\omega_\ell^{\frac{\ell-1}{2}}$.
 \item
 If $k \leq \ell+1$, then the following statements are true.
 \begin{itemize}
 \item
 If $\iota_\ell(a(\ell)) \in \bar{\Z}_\ell^\times$, then $k=\frac{\ell+1}{2}.$
 \item
  If $\iota_\ell(a(\ell)) \not \in \bar{\Z}_\ell^\times$, then $k=\frac{\ell+3}{2}.$
 \end{itemize}
 \end{enumerate}
 \end{lemma}
 \begin{proof}
 A standard argument (see e.g. the discussion before \cite[Definition $9.6.10$]{Diamond-Shurman}) shows that the character $\varphi$ has finite image with order coprime to $\ell$. If $\varphi$ is unramified, it must also factor through $\Gal(\Q(\zeta_{\ell^\infty})/\Q)$, where $\Q(\zeta_{\ell^\infty})$ is the $\ell$-adic cyclotomic extension.
 These two observations show that $\varphi$ factors through $\Gal(\Q(\zeta_\ell)/\Q)$.
 From part $(1)$ of Theorem~\ref{bigGalThm}, we also have
\[
\det \bar{\rho}_f=\omega^{k-1}_\ell=\det \bar{\rho}_g=\varphi^2\det \bar{\rho}_f,
\]
which means that $\varphi$ is quadratic.

Thus, to prove that $\varphi=\omega^{\frac{\ell-1}{2}}_\ell$,
we only need to show that $\varphi$ is unramified outside of $\ell$; this will show that $\varphi=\omega^{i}_\ell$ where $1 \leq i \leq \ell-2$, and we must then have $i=\frac{\ell-1}{2}$ since $\varphi$ is quadratic.
To this end, let $p \neq \ell$ be prime.
When $Q \nmid C_{\psi}$, parts $(1)$ and $(2)$ of Theorem~\ref{bigGalThm} show that
\begin{equation}
\bar{\rho}_{f}|_{I_{p}} \cong \left(\begin{matrix}1 & * \\0 & 1 \end{matrix}\right).
\end{equation}
The same conclusion holds for $\bar{\rho}_g |_{I_p}$, and it then follows from $\bar{\rho}_g \cong \bar{\rho}_f \otimes \varphi$ that $\varphi$ is unramified in this~case.

Assume that $Q \mid C_{\psi}$ and let $\sigma \in I_p$. Since we have $\bar{\rho}_g \cong \bar{\rho}_f \otimes \varphi$ 
it follows from part $(2)$ of 
Theorem $2.1$   
that 
$\operatorname{tr}(\bar{\rho}_g(\sigma))=\operatorname{tr}(\bar{\rho}_f(\sigma))$ and $\operatorname{tr}(\bar{\rho}_g(\sigma))=\varphi(\sigma)\operatorname{tr}(\bar{\rho}_f(\sigma))$. Since we have $\psi(\sigma) \neq -1$, part $(2)$ of Theorem~\ref{bigGalThm} also shows that $\varphi(\sigma) \neq -1$; the fact that $\varphi$ is quadratic further implies that $\varphi(\sigma)=1$. This completes the proof that $\varphi$ is unramified.

The proof of part $(2)$ follows as in the proof of \cite[Lemma $3.1$]{AAT} by using part $(3)$ of Theorem~\ref{bigGalThm} in a similar manner.
 \end{proof}
 The next lemma is an analogue of \cite[Lemma $3.2$]{AAT}. Since the proof follows as in their work using Lemma~\ref{3.1 analogue} in a similar fashion, we omit the details.
 \begin{lemma}\label{3.2 analogue}
 Let $\psi$ be a Dirichlet character modulo $N$ which we regard as a Galois character.
 Let $f=\sum a(n)q^n \in S_k(N,\psi)$ be a normalized  eigenform. Suppose that $\ell \nmid \phi(N)$ and that $\psi(\sigma) \neq -1$ for all $\sigma \in G_\Q$. Assume that $2 \leq k \leq \ell+1$ and that there is a prime $Q$  such that $Q \mid N$, $Q^{k-1} \not \equiv Q^{\pm 1} \pmod{\ell}$, and $f \in S_k^{\operatorname{new }Q}(N)$. Then the following statements hold.

 \begin{enumerate}
 \item
 $\bar{\rho}_f$ is irreducible.
 \item
 Assume that there is a quadratic extension $K/\Q$ such that $\bar{\rho}_f \sl_{G_K}$ is reducible. 
 Then $\bar{\rho}_f \cong \bar{\rho}_f \otimes \omega_\ell^{\frac{\ell-1}{2}}$ and
 \begin{itemize}
 \item 
 If $\iota_\ell(a(\ell)) \in \bar{\Z}^\times_\ell$, then $k=\frac{\ell+1}{2}$.
 \item
 If $\iota_\ell(a(\ell)) \not \in \bar{\Z}^\times_\ell$, then $k=\frac{\ell+3}{2}$.
 \end{itemize}
 \end{enumerate}
 \end{lemma}
 \subsection{Sufficient conditions for Suitability}
Here we show that suitability holds for many spaces of forms. (see also \cite[Proposition $3.3$]{AAT} and \cite[Proposition $7.2$]{AAD}).
\begin{proposition}\label{3.3 analogue}
Suppose that $r$ is an odd integer.
 Suppose that $N$ is a squarefree, odd, positive integer with $\ell \nmid N$. Assume that $3 \nmid N$ if $3 \nmid r$.
Let $\psi$ be a Dirichlet character modulo $N$. Let $k$ be an even, positive integer.
Then $(k,\ell)$ is suitable for every triple $(N,\psi,r)$ if the following conditions hold.
\begin{enumerate}
\item
$\psi(\sigma) \neq -1$ for all $\sigma \in G_\Q$.
\item 
$\ell \nmid \phi(N)$.
\item
$k \leq \ell+1$.
\item
$2^{k-1} \not \equiv 2^{\pm 1} \pmod\ell$.
\item
$k \neq \frac{\ell+1}2,
\frac{\ell+3}2$.
\item
$\frac{\ell+1}{(\ell+1,k-1)}, \frac{\ell-1}{(\ell-1,k-1)} \geq 6$.
\end{enumerate}
 When $\ell > 5k-4$, we always have conditions $(3)$, $(5)$ and $(6)$.
\end{proposition}

\begin{proof}
The final assertion is easy to check.
We also assume that $(r,6)=1$; the proof is similar when $(r,6)=3$. Suppose
that $f=\sum a(n)q^{n} \in S^{\new 2,3}_{k}(6N,\psi^2,\ep_{2,r,\psi},\ep_{3,r,\psi})$ is a newform.
From \cite[Theorem $2.47(b)$]{DDT97}, we see that there are four possibilities for the image of $\bar{\rho}_{f}$:
\begin{enumerate}
\item
$\bar{\rho}_{f}$ is reducible.
\item
$\bar{\rho}_{f}$ is dihedral, i.e. $\bar{\rho}_{f}$ is irreducible but $\bar{\rho}_{f} \sl_{G_{K}}$ is reducible for some quadratic $K/\Q$.
\item
$\bar{\rho}_{f}$ is exceptional, i.e. the projective image of $\bar{\rho}_{f}$ is conjugate to one of $A_{4}$, $S_{4}$, or $A_{5}$.
\item
The image of $\bar{\rho}_{f}$ contains a conjugate of $\SL_2(\mathbb{F}_\ell )$.
\end{enumerate}
Our strategy is to rule out the first $3$ cases. 
By 
conditions $(1)$, $(2)$, and $(4)$ together with part $(1)$ of Lemma~\ref{3.2 analogue}, we see that $\bar{\rho}_{f}$ is irreducible. 
By condition $(5)$ and part $(2)$ of Lemma~\ref{3.2 analogue}, we conclude that $\bar{\rho}_{f}$ is not dihedral.

 To rule out the exceptional case, it suffices to show that the projective image contains an element of order $\geq 6$.
 Suppose that $\iota_\ell (a(\ell)) \in \bar{\Z}^{\times}_\ell $.
  By part $(3)$ of Theorem~\ref{bigGalThm}, we know that
  \begin{equation}
  \rho_{f}\sl_{I_\ell } \cong \pMatrix{\chi_\ell ^{k-1}}{*}{0}{1}.
  \end{equation}
Since $\omega_\ell $ has order $\ell-1$, the projective image of $\bar{\rho}_{f}$ contains an element of order $\geq \frac{\ell-1}{(\ell-1,k-1)} \geq 6$.
If $\iota_\ell (a(\ell)) \not \in \bar{\Z}^{\times}_\ell$, then part $(3)$ of Theorem~\ref{bigGalThm} implies that
  \begin{equation}
  \bar{\rho}_{f} \cong \pMatrix{\omega^{k-1}_2}{0}{0}{\omega'^{k-1}_2}.
  \end{equation}
  Since $\omega_2/\omega'_2$ has order $\ell+1$, we conclude by condition $(6)$ that the projective image of $\bar{\rho}_{f}$ contains an element of order $\frac{\ell+1}{(\ell+1,k-1)} \geq 6$. 
  This completes the proof.
\end{proof}
\subsection{Crucial ingredients for proving Theorem~\ref{thm:cong1}}. Here we collect the results necessary for proving the main technical result which is essential for the proof of Theorem~\ref{thm:cong1}. The most important result is Proposition~\ref{nonuniformconj}, which hinges on the next result, which roughly states that normalized eigenforms with sufficiently disjoint projective modulo $\ell$ Galois representations cut out similarly disjoint field extensions; it also generalizes \cite[Lemma $3.7$]{AAT} to spaces of forms with arbitrary character. 
 \begin{lemma}\label{technical}
 Assume that $N$ is a positive integer and that $\psi$ is a Dirichlet character modulo $N$ with conductor $C_\psi$.
Let $f_1,...,f_s \in S_{k}(N,\psi)$ be normalized eigenforms. Assume for each $i \in \{1,...,s\}$ that $\bar{\rho}_{f_{i}}(G_\Q)$ contains a conjugate of $\SL_2(\F_\ell)$, and let $M_i=\bar{\Q}^{\operatorname{ker}(\bar{\rho}_{f_i})}$. Then the following statements are true.
\begin{enumerate}
\item 
For each $i \in \{1,...,s\}$, there exists an extension $K_i/\Q$ of degree at most $2$ in $M_i$ and a finite extension $\F_{i}/\F_\ell$ such that $\bar{\rho}_{f_i}(G_{K_i(\zeta_{C_\psi\ell})})$ is conjugate to $\SL_2(\F_{i})$.
\item 
Further assume for each $\tau \in \Gal(\bar{\F}_\ell/\F_\ell)$ that $r_{f_i} \not \cong {}^\tau r_{f_j}$ for each $i,j \in \{1,...,s\}$ with $i \neq j$. Let $K_i$ and $\F_i$ be defined as in part (1), and  set $M:=M_1...M_s$ and $K:=K_1...K_s$. Then 
$\Gal(M(\zeta_{C_{\psi}\ell})/K(\zeta_{C_{\psi}\ell})) \cong \prod^{s}_{i=1} \SL_2(\F_{i})$.
\end{enumerate}
 \end{lemma}
\begin{proof}
For each $i \in \{1,\dots,s\}$, since $\bar{\rho}_{f_i}(G_\Q)$ contains a conjugate of $\SL_2(\F_\ell)$, we see from \cite[Theorem $2.47(b)$]{DDT97} that there exists a finite extension $\F_{i}/\F_\ell$ such that $r_{f_i}(G_\Q)$ is conjugate to either $\text{PSL}_2(\F_{i})$ or $\text{PGL}_2(\F_{i})$. 
Upon replacing $\bar{\rho}_{f_i}$ by a conjugate, we assume that $r_{f_i}(G_\Q)=\text{PSL}_2(\F_{i})$ or $\text{PGL}_2(\F_{i})$. Now define $K_i/\Q$ to be the extension of degree at most $2$ in $M_i$ such that $r_{f_i}(G_{K_i})=\text{PSL}_2(\F_{i})$. Note that $\text{PSL}_2(\F_{i})$ is a simple group since $\ell \geq 5$. Since $K_i(\zeta_{C_{\psi}\ell})/K_i$ is an abelian extension, we see that $r_{f_i}(G_{K_i(\zeta_{C_{\psi}\ell})})=\PSL_2(\F_{i})$. This fact and  part $(1)$ of Theorem~\ref{bigGalThm} imply that $\bar{\rho}_{f_i}(G_{K_i(\zeta_{C_{\psi}\ell})}) =\SL_2(\F_{i})$.

The proof of part $(2)$ is very similar to the proof of part  $(2)$ of \cite[Lemma 3.7]{AAT}, so we sketch the details. Let $L_i:=M_i^{\text{ker}(r_{f_i})}$. As in the proof of part $(2)$ of \cite[~Lemma $3.7$]{AAT}, the following three claims hold.

\begin{itemize}
\item[(a)] 
$\Gal(L_iK(\zeta_{C_\psi\ell})/K(\zeta_{C_\psi\ell})) \cong \text{PSL}_2(\F_{i})$.
\item[(b)]
$\Gal(M_iK(\zeta_{C_\psi\ell})/K(\zeta_{C_\psi\ell})) \cong \text{SL}_2(\F_{i})$.
\item[(c)]
For $i,j \in \{1,\dots,s\}$ with $i \neq j$, we have $L_iK(\zeta_{C_\psi\ell}) \cap L_jK(\zeta_{C_\psi\ell})=K(\zeta_{C_\psi\ell})$.
\end{itemize}

The result follows by applying the above three claims and \cite[Lemma $3.6$]{AAT} as in the proof of \cite[Lemma $3.7$]{AAT}. More specifically, these tools allow us to prove by induction on $j \in \{1,\dots,s\}$ that
\[
\Gal(M_1\cdots M_jK(\zeta_{C_\psi\ell})/K(\zeta_{C_\psi\ell})) \cong \prod^j_{i=1} \Gal(M_iK(\zeta_{C_\psi\ell})/K(\zeta_{C_\psi\ell})) \cong \prod^j_{i=1} \SL_2(\F_{i}).
\]
\end{proof}
The next result, a generalization of \cite[Proposition $3.8$]{AAT}, is the crucial new ingredient which allows us to prove Theorem~\ref{thm:cong1}.
 \begin{proposition}\label{nonuniformconj}
 Suppose that $N \in \Z^+$ and that $\psi$ is a Dirichlet character modulo $N$.
 Let $f_1,...,f_s \in S_{k}(N,\psi)$ be normalized eigenforms such that for each $i \in \{1,\dots,s\}$, $\bar{\rho}_{f_{i}}(G_\Q)$ contains a conjugate of $\SL_2(\F_\ell)$.
 Then for every $\gamma \in \SL_2(\F_\ell)$, there exists $\sigma \in \Gal(\bar{\Q}/\Q(\zeta_{\ell}))$ such that for all $i \in \{1,...,s\}$, we have either $\bar{\rho}_{f_i}(\sigma)\thicksim \gamma$ or  $\bar{\rho}_{f_i}(\sigma) \thicksim -\gamma$. In particular, we have $\bar{\rho}_{f_i}(\sigma^2)\thicksim \gamma^2$.
 \end{proposition}
\begin{proof}
Let $S=\{f_1,\dots,f_s\}$.
For $f,g \in S$, we write $f \simeq g$ if there exists $\tau \in \Gal(\bar{\F}_\ell/\F_\ell)$ such that $r_{f_i} \cong {}^\tau r_{f_j}$. 
Let $S'=\{g_1,\dots,g_t\}$ be a complete set of representatives for $\simeq$.
It then follows from Lemma~\ref{technical} that there exists $\sigma \in \Gal(\bar{\Q}/\Q(\zeta_{C_{\psi}\ell}))$ such that $\bar{\rho}_g(\sigma) \thicksim \gamma$ for all $g \in S'$.

Now let $f \in S$. 
By definition, there exists $g \in S'$ such that $f \simeq g$. 
Thus, for some $\tau \in \Gal(\bar{\F}_\ell/\F_\ell)$ we can define a character $\varphi: G_\Q \rightarrow \bar{\F}^\times_\ell$ via $\varphi(\sigma)=\bar{\rho}_f(\sigma){}^\tau\bar{\rho}_g(\sigma)^{-1}$. 
It follows from $\bar{\rho}_g(\sigma) \thicksim \gamma$ and part $(1)$ of
Theorem~\ref{bigGalThm} that $\det \bar{\rho}_{g}(\sigma)=\det \bar{\rho}_f(\sigma)=\varphi(\sigma)^2 \det \bar{\rho}_g(\sigma)$. 
Thus, we have $\varphi(\sigma)^2=1$ and the result follows.
\end{proof}

\subsection{Main technical results}
 We now prove the main technical result which is used in the proof of Theorem~\ref{thm:cong1}.
\begin{theorem}\label{3.9 analogue}
 Suppose that  $\ell \geq 5$ is prime and that $r$ is an odd integer.
 Suppose that $N$ is a squarefree, odd, positive integer with $\ell \nmid N$. Assume that $3 \nmid N$ if $3 \nmid r$. Let $\psi$ be a Dirichlet character modulo $N$. Let $k$ be an even positive integer such that $(k,\ell)$ is suitable for $(N,\psi,r)$ and 
 let $m \geq 1$ be an integer.
  
  If $(r,6)=1$, then there exists a positive density set $S$ of primes such that if $p \in S$, then $p \equiv 1 \pmod{\ell^{m}}$ and $f \sl T_p \equiv -f \pmod{\ell^{m}}$ for each newform $f \in S^{\new 2,3}_{k}(6N,\psi^2,\ep_{2,r,\psi},\ep_{3,r,\psi})$. If $(r,6)=3$, then we have the same result for $S^{\new 2}_{k}(2N,\psi^2,\ep_{2,r,\psi})$. 
 \end{theorem}
 \begin{proof}
 We assume that $(r,6)=1$; we omit the details when $(r,6)=3$ since the proof is nearly identical in that case.
 Choose a number field $K$ which contains all of the coefficients of all of the newforms $f \in S^{\new 2,3}_{k}(6N,\psi^2,\ep_{2,r,\psi},\ep_{3,r,\psi})$. 
 Recall that we have fixed an embedding $\iota_\ell: \bar{\Q}\rightarrow \bar{\Q}_\ell$; if $\lambda$ is the prime of $K$ induced by $\iota_\ell$, then let $K_\lambda$ be the completion of $K$ at $\lambda$ with ring of integers $\mathcal{O}_\lambda$ and residue field $\mathbb{F}=\mathcal{O}_\lambda/\lambda$. Then for any newform $f \in S^{\new 2,3}_{k}(6N,\psi^2,\ep_{2,r,\psi},\ep_{3,r,\psi})$, the Galois representations $\rho_f$ and $\bar{\rho}_f$ of Theorem~\ref{bigGalThm} are defined over $\mathcal{O}_\lambda$ and $\mathbb{F}$, respectively.

 By the definition of suitability and Proposition~\ref{nonuniformconj},
 there exists $\sigma \in \Gal(\bar{\Q}/\Q(\zeta_{C_\psi\ell}))$ such that $\bar{\rho}_f(\sigma)$ $\thicksim$ $\pMatrix{1}{1}{-1}{0}^2=\pMatrix{0}{1}{-1}{-1}$ for every newform $f \in S^{\operatorname{new} 2,3}_k(6N, \psi^2,\epsilon_{2,r,\psi},\epsilon_{3,r,\psi})$. 
 Thus, the characteristic polynomial of $\rho_f(\sigma)$ is congruent to $x^2+x+1 \pmod{\lambda}$.
 If $w$ is a positive integer, we argue as in the proof of \cite[Theorem $1.5$]{AAT} to see that the characteristic polynomial of $\rho_f(\sigma^{\ell^{w-1}})$ is congruent to $x^2+x+1 \pmod{\lambda^w}$. We now apply the Chebotarev density theorem and Theorem~\ref{bigGalThm} as in the proof of \cite[Theorem $1.5$]{AAT}. More precisely, we conclude that for every positive integer $w$,
 there exists a positive density set $S_w$ of primes such that if $p \in S_w$, then $p \equiv 1 \pmod{\ell^w}$ and we have
 \[
 f \sl T_p \equiv -f \pmod{\lambda^w}
 \]
 for each newform $f \in S^{\new 2,3}_{k}(6N,\psi^2,\ep_{2,r,\psi},\ep_{3,r,\psi})$.

 The rest of the proof will follow as in the proof of \cite[Theorem $7.3$]{AAD} once we know that $S^{\operatorname{new} 2,3}_{k}(6N,\psi^2,\ep_{2,r,\psi},\ep_{3,r,\psi})$ has a basis with integer coefficients. To prove this claim, let $\{f_1,\dots,f_s\}$ be the set of newforms in this space. 
 By arguing as in \cite[Theorem $5.8.3$]{Diamond-Shurman}, we see that this space has a basis consisting of modular forms of the form $f_i \sl V_d$ where $d \mid 6N/M$ and $f_i$ is a newform of level $M$. 
 Our claim follows from a standard argument \cite[$\mathsection$~$6.1$]{Diamond-Shurman} once we know that this basis is stable under the action of $\Gal(\bar{\Q}/\Q)$.
 This follows from the fact that the Galois action commutes with $V_d$ and that the Atkin-Lehner eigenvalues at $2$ and $3$ are determined by the $2^{\operatorname{nd}}$ and $3^{\operatorname{rd}}$ coefficients, respectively, of the $f_i$. 
 Our proof is complete, given that these are integers by \cite[~~Corollary $3.2$]{AAD}. 
 \end{proof}

We also record the following analogue of \cite[Theorem $7.4$]{AAD}. It follows from the last assertion of part $(2)$ of Theorem~\ref{bigGalThm} using the same argument.

\begin{theorem}\label{4.2 analogue}
Let $k \in \Z^+$. Suppose that $\ell \geq 5$ is prime and that there exists an integer $a$ for which $2^a \equiv -2 \pmod{\ell}$.
Let $m \geq 1$ be an integer.
Let $N \in \Z^+$ be odd and squarefree with $\ell \nmid N$, and $3 \nmid N$ if $3 \nmid r$.
Assume that $\psi$ is a Dirichlet character modulo $N$.

If $(r,6)=1$, then there exists a positive density set $S$ of primes such that if $p \in S$, then $p \equiv -2 \pmod{\ell^m}$ and for each newform $f \in S^{\operatorname{new} 2,3}_k(6N, \psi^2,\epsilon_{2,r,\psi},\epsilon_{3,r,\psi})$, we have
\[
f \sl T_p \equiv -(-\epsilon_{2,r,\psi})^ap^{\frac{k}{2}-1}f \pmod{\ell^m}.
\]
The same result holds for $S_k^{\operatorname{new} 2}(2N, \psi^2,\epsilon_{2,r,\psi})$ if $(r,6)=3$.

\end{theorem}
 \section{Proof of Theorems~\ref{thm:cong1}-\ref{thm:cong2}}
 Let $\ell \geq 5$ be prime and $r$ be an odd integer. Let $m, \lambda \geq 1$ be an integer.
 Suppose that $N$ is a squarefree, odd, positive integer with $\ell \nmid N$. Assume that $3 \nmid N$ if $3 \nmid r$.
 Let $\psi$ be a Dirichlet character modulo $N$ and suppose that
 \begin{equation}
 F(z)=\displaystyle \sum_{n \equiv r \spmod{24}} a(n)q^{\frac{n}{24}} \in S_{\lambda+\frac{1}2}(N, \psi\nu_\eta^{r}).
 \end{equation}
 Furthermore, if $\lambda=1$, suppose that $F \in S^c_{3/2}(N,\psi\nu_\eta^r)$. 
 
 For each positive, squarefree integer $t$, recall that
 $\mathcal{S}_{t}(F) \in S^{\new 2,3}_{2\lambda}(6N,\psi^2,\ep_{2,r,\psi},\ep_{3,r,\psi})$ if $(r,6)=1$ and  $\mathcal{S}_{t}(F) \in S^{\new 2}_{2\lambda}(2N,\psi^2,\ep_{2,r,\psi})$ if $(r,6)=3$.
 As $t$ ranges over all squarefree integers, there are only finitely many non-zero possibilities for $\mathcal{S}_t(F)$ modulo $\ell^m$; let $\{g_1,\dots,g_k\}$ be a set of representatives for these possibilities.  Let
  $\{f_1,\dots,f_s\}$ be the set of newforms in $S^{\new 2,3}_{2\lambda}(6N,\psi^2,\ep_{2,r,\psi},\ep_{3,r,\psi})$ if $(r,6)=1$ and in $S^{\new 2}_{2\lambda}(2N,\psi^2, \ep_{2,r,\psi})$ if $(r,6)=3$.
 Write
 \begin{equation}\label{linearcombination}
 g_{j}=\sum^{s}_{i=1}\sum_{d \mid N}c_{i,j,d}f_{i} \sl V_d
 \end{equation}
 for $j \in \{1,\dots,k\}$ with $c_{i,j,d} \in \bar{\Q}$.
 Choose $c \geq 0$ such that $\ell^{c}c_{i,j,d}$ is integral at all primes above $\ell$ for all $c_{i,j,d}$.
 Recall also the character $\chi^{(r)}$ in \eqref{chir}.
 We require two short lemmas (which have proofs very similar to the proofs of \cite[Lemmas $7.5-7.6$]{AAD}, using \eqref{eq:heckedef24}, \eqref{f0iffShimura0}, and \eqref{equivariance}).

\begin{lemma}\label{5.1 analogue}
Suppose that $p, \ell \geq 5$ are prime, and that  $r$ is an odd integer.
Let $N$ be a squarefree, odd, positive integer such that $p \nmid N$, $\ell \nmid N$, and $3 \nmid N$ if $3 \nmid r$.
Let $\psi$ be a Dirichlet character modulo $N$ and
suppose that $F \in S_{\lambda+1/2}(N, \psi\nu_\eta^{r})$. 
Let $m \geq 1$ be an integer, and 
let the forms $f_{i}$ with $i \in \{1,\dots,s\}$ and the integer $c$ be defined as above. If  $\lambda_{p}$ is an integer with the property that $f_i \sl T_p \equiv \lambda_{p}f_{i} \pmod{\ell^{m+c}}$ for all $i \in \{1,\dots,s\}$, then
\begin{equation}
F \sl T_{p^2} \equiv \chi^{(r)}(p)\lambda_{p}F \pmod{\ell^m} .
\end{equation}
\end{lemma}

\begin{lemma}\label{5.2 analogue}
Suppose that  $p, \ell \geq 5$ are prime and that $r$ is an odd integer.
Let $N$ be a squarefree, odd, positive integer such that $p \nmid N$, $\ell \nmid N$, and $3 \nmid N$ if $3 \nmid r$.
Let $m \geq 1$ be an integer.
Let  $\psi$ be a Dirichlet character modulo $N$ and suppose that 
\begin{equation}\label{exp24thpowers}
F=\displaystyle\sum_{n \equiv r \spmod{24}} a(n)q^{\frac{n}{24}} \in S_{\lambda+\frac{1}2}(N,\psi\nu_\eta^{r}).
\end{equation} 
Suppose that  there exists $\alpha_{p} \in \{\pm 1\}$ with
\begin{equation}
F\sl T_{p^2} \equiv \alpha_{p}p^{\lambda-1}F \pmod{\ell^m} .
\end{equation}
Then we have
\begin{equation}
a(p^2n) \equiv 0 \pmod{\ell^m}  \ \ \ \text{ if } \ \ \ \pmfrac{n}{p} =\alpha_{p} \pmfrac{12}p \pmfrac{-1}p^{\frac{r-1}{2}}\psi(p).
\end{equation}
\end{lemma}

We now prove Theorems~\ref{thm:cong1}-\ref{thm:cong2}.

\begin{proof}[Proof of Theorem~\ref{thm:cong1}]
Suppose that $\ell \geq 5$ is a prime.
Let $c$ be an integer as in Lemma~\ref{5.1 analogue}.
If $(r,6)=1$, then Theorem~\ref{3.9 analogue} implies that there exists a positive density set $S$ of primes such that if $p \in S$, then we have $p \equiv 1 \pmod{\ell^m}$, $p \nmid 6N$, and
\[
f \sl T_p \equiv -f \pmod{\ell^{m+c}}
\]
for each newform $f \in  S^{\new 2,3}_{2\lambda}(6N,\psi^2,\ep_{2,r,\psi},\ep_{3,r,\psi})$; when $(r,6)=3$, the same conclusion holds for $S^{\operatorname{new} 2}_{2\lambda}(2N, \psi^2,\ep_{2,r,\psi})$. For such $p$, it follows from Lemma~\ref{5.1 analogue} that
\[
F \sl T_{p^2} \equiv \chi^{(r)}(p)F \pmod{\ell^m}.
\]
The result follows from Lemma~\ref{5.2 analogue}.
\end{proof}

 \begin{proof}[Proof of Theorem~\ref{thm:cong2}]
 Suppose that $\ell \geq 5$ is a prime such that $2^{a} \equiv -2 \pmod{\ell} $ for some integer $a$. Let $c$ be an integer as in Lemma~\ref{5.1 analogue}.
 If  $(r,6)=1$, then
 by Theorem~\ref{4.2 analogue}, there exist  $\beta \in \{\pm 1\}$ and a positive density set $S$ of primes such that if $p \in S$, then $p \equiv -2 \pmod{\ell^{m}}$, $p \nmid 6N$, and for each newform $f \in S^{\new 2,3}_{2\lambda}(6N,\psi^2, \ep_{2,r,\psi},\ep_{3,r,\psi})$, we have
 \begin{equation}
 f \sl T_p \equiv \beta p^{\lambda-1}f \pmod{\ell^{m+c}}.
 \end{equation}
 If $(r, 6)=3$ the same result holds for $S^{\new 2}_{2\lambda}\(2N,\psi^2,\ep_{2,r,\psi}\)$.
 
 In either case, for  such $p$, Lemma~\ref{5.1 analogue} implies that
 \begin{equation}
F \sl T_{p^2} \equiv \beta\chi^{(r)}(p)p^{\lambda-1}F \pmod{\ell^m} .
\end{equation}
The result follows from Lemma~\ref{5.2 analogue}. 
 \end{proof}

\section{Acknowledgements}
The author thanks Patrick Allen for several helpful discussions, and Olivia Beckwith for helpful comments.
\bibliographystyle{amsalpha}
\bibliography{Quadraticcongruences}
\end{document}